\documentclass[reqno]{amsart}

\usepackage{amscd, amssymb, amsthm, mathrsfs, amsmath, amsrefs}
\usepackage[all]{xy}

\usepackage{graphicx}
\usepackage{txfonts}
\usepackage[all]{xy}
\usepackage{bbm}
\usepackage{lscape}
\usepackage{tikz}
\usetikzlibrary{matrix,arrows}
\usepackage[bookmarks]{hyperref}
\usepackage{enumitem}

\theoremstyle{plain}
\newtheorem{thm}[subsection]{Theorem}

\newtheorem{prop}[subsection]{Proposition}

\newtheorem{lem}[subsection]{Lemma}
\newcommand{\comment}[1]{}

\theoremstyle{definition}
\newtheorem{defi}[subsection]{Definition}

\theoremstyle{remark}
\newtheorem{rem}[subsection]{Remark}

\theoremstyle{definition}
\newtheorem{exa}[subsection]{Example}

\newcommand{\Alg}{\mathrm{Alg}}
\newcommand{\Mod}{\mathrm{Mod}}
\renewcommand{\S}{\mathbb{S}}
\newcommand{\Z}{\mathbb{Z}}
\newcommand{\N}{\mathbb{N}}
\newcommand{\Hom}{\mathrm{Hom}}
\newcommand{\Map}{\mathrm{Map}}
\newcommand{\colim}{\operatornamewithlimits{colim}}
\newcommand{\Zo}{{\mathbb{Z}_{\geq 0}}}
\newcommand{\sd}{\mathrm{sd}}
\newcommand{\Ex}{\mathit{Ex}}
\newcommand{\ind}{\mathrm{ind}}
\newcommand{\op}{\mathrm{op}}
\newcommand{\mult}{\mathbf{m}}
\newcommand{\id}{\mathrm{id}}
\newcommand{\fS}{\mathfrak{S}}
\newcommand{\ev}{\mathrm{ev}}
\newcommand{\pr}{\mathrm{pr}}

\newcommand{\inc}{\mathrm{incl}}
\newcommand{\tilT}{\widetilde{T}}
\newcommand{\scrK}{\mathscr{K}}
\newcommand{\Born}{\mathrm{Born}}
\newcommand{\SigmaHo}{{\Sigma\mathrm{Ho}}}

\usepackage{tensor}

\newcommand{\fC}{\mathfrak{C}}
\newcommand{\scrU}{\mathscr{U}}
\newcommand{\scrE}{\mathscr{E}}
\renewcommand{\bul}{\bullet}

\newcommand{\Minf}{{M_\infty}}

\newcommand{\lar}[1]{\vec{#1}}

\newcommand{\tria}{\mathscr{T}}
\newcommand{\scrP}{\mathscr{P}}

\newcommand{\sk}{\mathrm{sk}}

\newcommand{\Algl}{{\mathrm{Alg}_\ell}}
\newcommand{\Set}{\mathrm{Set}}

\newcommand{\DRF}{{D(\mathfrak{F}_{\mathrm{spl}})}}

\begin{document}

\title{The homotopy groups of the simplicial mapping space between algebras}
\author{Emanuel Rodr\'\i guez Cirone}
\email{ercirone@dm.uba.ar}
\address{Dep. Matem\'atica-IMAS, FCEyN-UBA\\ Ciudad Universitaria Pab 1\\
	1428 Buenos Aires\\ Argentina}

\begin{abstract}
Let $\ell$ be a commutative ring with unit. To every pair of $\ell$-algebras $A$ and $B$ one can associate a simplicial set $\Hom(A,B^\Delta)$ so that $\pi_0\Hom(A,B^\Delta)$ equals the set of polynomial homotopy classes of morphisms from $A$ to $B$. We prove that $\pi_n\Hom(A,B^\Delta)$ is the set of homotopy classes of morphisms from $A$ to $B^{\fS_n}_\bul$, where $B^{\fS_n}_\bul$ is the ind-algebra of polynomials on the $n$-dimensional cube with coefficients in $B$ vanishing at the boundary of the cube. This is a generalization to arbitrary dimensions of a theorem of Corti\~nas-Thom, which addresses the cases $n\leq 1$. As an application we give a simplified proof of a theorem of Garkusha that computes the homotopy groups of his matrix-unstable  algebraic $KK$-theory space in terms of polynomial homotopy classes of morphisms.
\end{abstract}


\maketitle

\section{Introduction}

Algebraic $kk$-theory was constructed by Corti\~nas-Thom in \cite{cortho}, as a completely algebraic analogue of Kasparov's $KK$-theory. It is defined on the category $\Algl$ of associative, not necessarily unital algebras over a fixed unital commutative ring $\ell$. It consists of a triangulated category $kk$ endowed with a functor $j:\Algl\to kk$ that satisfies the following properties:
\begin{enumerate}
	\item[(H)]\label{item1} \emph{Homotopy invariance.} The functor $j$ is polynomial homotopy invariant.
	\item[(E)]\label{item2} \emph{Excision.} Every short exact sequence of $\ell$-algebras that splits as a sequence of $\ell$-modules gives rise to a distinguished triangle upon applying $j$.
	\item[(M)]\label{item3} \emph{Matrix stability.} For any $\ell$-algebra $A$ we have $j(\Minf A)\cong j(A)$, where $\Minf A$ denotes the algebra of finite matrices with coefficients in $A$ indexed by $\N\times\N$.
\end{enumerate}
This functor $j$ is moreover universal with the above properties: any other functor from $\Algl$ into a triangulated category satisfying (H), (E) and (M) factors uniquely trough $j$. Another important property of $kk$-theory is that it recovers Weibel's homotopy $K$-theory: $kk(\ell,A)\cong KH_0(A)$; see \cite{cortho}*{Theorem 8.2.1}.

As a technical tool for defining algebraic $kk$-theory, Corti\~nas-Thom introduced in \cite{cortho}*{Section 3} a simplicial enrichment of $\Algl$. They associated a simplicial mapping space $\Hom_\Algl(A,B^\Delta)$ to any pair of $\ell$-algebras $A$ and $B$, and they defined simplicial compositions
\[\circ:	\Hom_\Algl(B,C^\Delta)\times\Hom_\Algl(A,B^\Delta)\to\Hom_\Algl(A,C^\Delta)\]
that make $\Algl$ into a simplicial category in the sense of \cite{quillen}*{Section II.1}. The homotopy category of this simplicial category is Gersten's homotopy category of algebras \cite{gersten}*{Section 1}. This means that there is a natural bijection
\[\pi_0\Hom_\Algl(A,B^\Delta)\cong [A,B],\]
where the right hand side denotes the set of polynomial homotopy classes of morphisms from $A$ to $B$. In the same vein, Corti\~nas-Thom showed in \cite{cortho}*{Theorem 3.3.2} that
\[\pi_1\Hom_\Algl(A,B^\Delta)\cong [A,B^{\fS_1}],\]
where $B^{\fS_1}$ denotes the ind-algebra of polynomials on $S^1=\Delta^1/\partial\Delta^1$  with coefficients in $B$ that vanish at the basepoint. The main result of this paper is the following generalization of the latter to arbitrary dimensions.

\begin{thm}[Theorem \ref{thm:bij}]\label{thm:main}
For any pair of $\ell$-algebras $A$ and $B$ and any $n\geq 0$ there is a natural bijection
\begin{equation}\label{eq:main}\pi_n\Hom_\Algl(A,B^\Delta)\cong [A,B^{\fS_n}],\end{equation}
where $B^{\fS_n}$ is the ind-algebra of polynomials on the $n$-dimensional cube with coefficients in $B$ vanishing at the boundary of the cube.
\end{thm}

To prove Theorem \ref{thm:main} one has to compare two different notions of homotopy for ind-algebra homomorphisms $A\to B^{\fS_n}$: simplicial homotopy on the left hand side of \eqref{eq:main} and polynomial homotopy on the right. Simplicial homotopy implies polynomial homotopy by \cite{garku}*{Hauptlemma (2)}. This key technical result of \cite{garku} ---of which Garkusha provides a beautiful constructive proof--- allows one to define a surjective function
\begin{equation}\label{eq:function}\pi_n\Hom_\Algl(A,B^\Delta)\to [A,B^{\fS_n}]\end{equation}
that turns out to be the desired bijection. We prove the injectivity of \eqref{eq:function} by showing that polynomial homotopy implies simplicial homotopy; this is done in Lemma \ref{lem:haupt}. This lemma follows immediately from the existence of the multiplication morphisms defined in section \ref{sec:multi}.

With methods different from those of Corti\~nas-Thom, Garkusha gave in \cite{garku} an alternative construction of $kk$-theory and moreover defined other universal bivariant homology theories of algebras. The latter are functors from $\Algl$ into some triangulated category that share properties (H) and (E) with algebraic $kk$-theory but satisfy different matrix-stability conditions. Garkusha showed in \cite{garku} that his bivariant homology theories are representable by spectra, and the simplicial mapping spaces between algebras are his main building blocks for these spectra. The idea of the isomorphism \eqref{eq:function} was already present in the proof of \cite{garku}*{Comparison Theorem A}, where he computed the homotopy groups of the matrix-unstable algebraic $KK$-theory space in terms of polynomial homotopy classes of morphisms. However, Garkusha used \cite{garku}*{Hauptlemma (3)} as a substitute for Lemma \ref{lem:haupt} when proving injectivity. This makes his proof rely on nontrivial techniques from homotopy theory such as the construction of a motivic-like model category of simplicial functors on $\Algl$. As an application of Theorem \ref{thm:main}, we give a simplified proof of \cite{garku}*{Comparison Theorem A} that uses no more homotopy theory than the definition of the homotopy groups of a simplicial set.

The rest of this paper is organized as follows. In section \ref{sec:conv} we fix notation and we recall the definition of polynomial homotopy and the details of the simplicial enrichment of $\Algl$. In section \ref{sec:main} we prove Lemma \ref{lem:haupt} and Theorem \ref{thm:main}. In section \ref{sec:spec} we apply Theorem \ref{thm:main} to give a simplified proof of \cite{garku}*{Comparison Theorem A}.

\section{Preliminaries}\label{sec:conv} Throughout this text, $\ell$ is a commutative ring with unit. We only consider associative, not necessarily unital $\ell$-algebras and we write $\Algl$ for the category of $\ell$-algebras and $\ell$-algebra homomorphisms. Simplicial $\ell$-algebras can be considered as simplicial sets using the forgetful functor $\Algl\to \Set$; this is usually done without further mention. The symbol $\otimes$ indicates tensor product over $\Z$.

\subsection{Categories of directed diagrams}\label{sec:directed}

Let $\fC$ be a category. A \emph{directed diagram} in $\fC$ is a functor $X:I\to\fC$, where $I$ is a filtering partially ordered set. We often write $(X,I)$ or $X_\bul$ for such a functor. We shall consider different categories whose objects are directed diagrams:

\subsubsection{Fixing the filtering poset} Let $I$ be a filtering poset. We will write $\fC^I$ for the category whose objects are the functors $X:I\to\fC$ and whose morphisms are the natural transformations.

\subsubsection{Varying the filtering poset} We will write $\lar{\fC}$ for the category whose objects are the directed diagrams in $\fC$ and whose morphisms are defined as follows: Let $(X,I)$ and $(Y,J)$ be two directed diagrams. A morphism from $(X,I)$ to $(Y,J)$ consists of a pair $(f,\theta)$ where $\theta:I\to J$ is a functor and $f:X\to Y\circ\theta$ is a natural transformation.

For a fixed filtering poset $I$, there is a faithful functor $a:\fC^I\to \lar{\fC}$ that acts as the identity on objects and that sends a natural transformation $f$ to the morphism $(f,\id_I)$.

\subsubsection{The category of ind-objects} The category $\fC^\ind$ of ind-objects of $\fC$ is defined as follows: The objects of $\fC^\ind$ are the directed diagrams in $\fC$. The hom-sets are defined by:
\[\Hom_{\fC^\ind}\left( (X,I),(Y,J)\right):=\lim_{i\in I}\colim_{j\in J}\Hom_\fC(X_i,Y_j)\]

There is a functor $\lar{\fC}\to\fC^\ind$ that acts as the identity on objects and that sends a morphism $(f,\theta):(X,I)\to(Y,J)$ to the morphism:
\[\left(f_i:X_i\to Y_{\theta(i)}\right)_{i\in I}\in \lim_{i\in I}\colim_{j\in J}\Hom_\fC(X_i,Y_j)\]

\subsection{Simplicial sets}\label{sec:simpli} The category of simplicial sets is denoted by $\S$; see \cite{hovey}*{Chapter 3}. Let $\Map(?,??)$ be the internal-hom in $\S$; we often write $Y^X$ instead of $\Map(X,Y)$.

\subsubsection{The iterated last vertex map}\label{subsec:lastvertex}
Let $\sd:\S\to\S$ be the subdivision functor. There is a natural transformation $\gamma:\sd\to \id_\S$ called the \emph{last vertex map} \cite{goja}*{Section III. 4}. For $X\in\S$, put $\gamma^1_X:=\gamma_X$ and define inductively $\gamma^n_X$ to be the following composite:
\[\xymatrixcolsep{5em}\xymatrix{
	\sd^nX=\sd(\sd^{n-1}X) \ar[r]^-{\gamma^1_{\sd^{n-1}X}} & \sd^{n-1}X \ar[r]^-{\gamma^{n-1}_X} & X}\]
It is immediate that $\gamma^n:\sd^n\to \id_\S$ is a natural transformation. Let $\sd^0:\S\to\S$ be the identity functor and let $\gamma^0:\sd^0\to\id_\S$ be the identity natural transformation.

\begin{lem}
	For any $p,q\geq 0$ and any $X\in\S$ we have:
	\[\gamma^{p+q}_X=\gamma^p_X\circ \sd^p\left(\gamma^q_X\right)=\gamma^p_X\circ \gamma^q_{\sd^pX}\]
\end{lem}
\begin{proof}
	It follows from a straightforward induction on $n=p+q$.
\end{proof}

\subsubsection{Simplicial cubes} Let $I:=\Delta^1$ and let $\partial I:=\{0,1\}\subset I$. For $n\geq 1$, let $I^n:=I\times\cdots\times I$ be the $n$--fold direct product and let $\partial I^n$ be the following simplicial subset of $I^n$:
\[\partial I^n:=\left[(\partial I)\times I\times\cdots\times I\right]\cup\left[I\times (\partial I)\times\cdots\times I\right]\cup\cdots\cup\left[I\times\cdots\times I\times (\partial I)\right]\]
Let $I^0:=\Delta^0$ and let $\partial I^0:=\emptyset$. We identify $I^{m+n}=I^m\times I^n$ and $\partial (I^{m+n})=\left[(\partial I^m)\times I^n\right]\cup\left[I^m\times (\partial I^n)\right]$ using the associativity and unit isomorphisms of the direct product in $\S$.

\subsubsection{Iterated loop spaces}\label{subsec:loop} Let $(X,*)$ be a pointed fibrant simplicial set. Recall from \cite{goja}*{Section I.7} that the loopspace $\Omega X$ is defined as the fiber of a natural fibration $\pi_X:PX\to X$, where $PX$ has trivial homotopy groups. By the long exact sequence associated to this fibration, we have pointed bijections $\pi_{n+1}(X,*)\cong\pi_n(\Omega X,*)$ for $n\geq 0$ that are group isomorphisms for $n\geq 1$. Iterating the loopspace construction we get:
\[\pi_0(\Omega^nX)\cong \pi_1(\Omega^{n-1}X,*)\cong \cdots \cong \pi_n(X,*)\]
Thus, $\pi_0\Omega^nX$ is a group for $n\geq 1$ and this group is abelian for $n\geq 2$. Moreover, a morphism $\varphi:X\to Y$ of pointed fibrant simplicial sets induces group homomorphisms $\varphi_*:\pi_0\Omega^n X\to \pi_0\Omega^n Y$ for $n\geq 1$. Let $\inc$ denote the inclusion $\partial I^n\to I^n$. It is easy to see that the iterated loop functor $\Omega^n$ on pointed fibrant simplicial sets can be alternatively defined by the following pullback of simplicial sets:
\begin{equation}\label{eq:Omega1}\begin{gathered}\xymatrix{\Omega^nX\ar[r]^-{\iota_{n,X}}\ar[d] & \Map(I^n,X)\ar[d]^-{\inc^*} \\
	\Delta^0\ar[r]^-{*} & \Map(\partial I^n,X) \\}\end{gathered}\end{equation}
We will always use the latter description of $\Omega^n$. Occasionally we will need to compare $\Omega^n$ for different integers $n$; for this purpose we will explicitely describe how the diagram \eqref{eq:Omega1} arises from successive applications of the functor $\Omega$. We start defining $\Omega X$ by the following pullback in $\S$:
\[\xymatrix{\Omega X\ar[r]^-{\iota_{1,X}}\ar[d] & \Map(I,X)\ar[d]^-{\inc^*} \\
	\Delta^0\ar[r]^-{*} & \Map(\partial I,X) \\}\]
For $n\geq 1$, define inductively $\iota_{n+1,X}:\Omega^{n+1}X\to\Map(I^{n+1},X)$ as the following composite:
\[\xymatrix{\Omega\left(\Omega^nX\right)\ar[r]^-{\iota_{1,\Omega^nX}} & \Map\left(I,\Omega^nX\right)\ar[r]^-{(\iota_{n,X})_*} & \Map\left(I,\Map(I^n,X)\right)\cong\Map\left(I^n\times I,X\right)}\]
It is easily verified that \eqref{eq:Omega1} is a pullback. Moreover, $\iota_{m+n,X}$ equals the following composite:
\[\xymatrixcolsep{2em}\xymatrix{\Omega^n\left(\Omega^mX\right)\ar[r]^-{\iota_{n,\Omega^mX}} & \Map\left(I^n,\Omega^mX\right)\ar[r]^-{(\iota_{m,X})_*} & \Map\left(I^n,\Map(I^m,X)\right)\cong\Map\left(I^m\times I^n,X\right)}\]
Thus, under the identification of diagram \eqref{eq:Omega1}, each time we apply $\Omega$ the new $I$-coordinate appears to the right.

\subsection{Simplicial enrichment of algebras}\label{subsec:simpenrich} We proceed to recall some of the details of the simplicial enrichment of $\Algl$ introduced in \cite{cortho}*{Section 3}. Let $\Z^\Delta$ be the simplicial ring defined by:
\[[p]\mapsto\Z^{\Delta^p}:=\Z[t_0,\dots ,t_p]/\langle 1-\textstyle\sum t_i\rangle\]
An order-preserving function $\varphi:[p]\to[q]$ induces a ring homomorphism $\Z^{\Delta^q}\to\Z^{\Delta^p}$ by the formula:
\[t_i\mapsto \sum_{\varphi(j)=i}t_j\]
Now let $B\in\Alg_\ell$ and define a simplicial $\ell$-algebra $B^\Delta$ by:
\begin{equation}\label{eq:defideltap}[p]\mapsto B^{\Delta^p}:=B\otimes \Z^{\Delta^p}\end{equation}
If $A$ is another $\ell$-algebra, the simplicial set $\Hom_\Algl(A,B^\Delta)$ is called the \emph{simplicial mapping space} from $A$ to $B$. For $X\in \S$, put $B^X:=\Hom_\S(X,B^\Delta)$; it is easily verified that $B^X$ is an $\ell$-algebra with the operations defined pointwise. When $X=\Delta^p$, this definition of $B^{\Delta^p}$ coincides with \eqref{eq:defideltap}. We have a natural isomorphism as follows, where the limit is taken over the category of simplices of $X$:
\[B^X\overset{\cong}\to  \displaystyle\lim_{\Delta^p\downarrow X}B^{\Delta^p}\]
For $A,B\in\Alg_\ell$ and $X\in \S$ we have the following adjunction isomorphism:
\[\Hom_\S(X,\Hom_{\Alg_\ell}(A,B^\Delta))\cong\Hom_{\Alg_\ell}(A,B^X)\]

\begin{rem}\label{rem:simpliberreta}
	Let $X$ and $Y$ be simplicial sets. In general $(B^X)^Y\not\cong B^{X\times Y}$ ---this already fails when $X$ and $Y$ are standard simplices; see \cite{cortho}*{Remark 3.1.4}.
\end{rem}

\begin{rem}\label{rem:mult} The simplicial ring $\Z^{\Delta}$ is commutative and hence the same holds for the rings $\Z^X=\Hom_\S(X,\Z^\Delta)$, for any $X\in\S$. Thus, the multiplication in $\Z^X$ induces a ring homomorphism $\mult_X:\Z^X\otimes\Z^X\to\Z^X$. Note that $\mult_X$ is natural in $X$.
\end{rem}

\subsection{Polynomial homotopy}\label{sec:alghtp} Two morphisms $f_0,f_1:A\to B$ in $\Algl$ are \emph{elementary homotopic} if there exists an $\ell$-algebra homomorphism $f:A\to B[t]$ such that $\ev_0\circ f=f_1$ and $\ev_1\circ f=f_0$. Here, $\ev_i$ stands for the evaluation $t\mapsto i$. Equivalently, $f_0$ and $f_1$ are elementary homotopic if there exists $f:A\to B^{\Delta^1}$ such that the following diagram commutes for $i=0,1$:
\[\xymatrixcolsep{3em}\xymatrixrowsep{2em}\xymatrix{A\ar[d]_-{f_i}\ar[r]^-{f} & B^{\Delta^1}\ar[d]^-{(d^i)^*} \\
	B\ar[r]^-{\cong} & B^{\Delta^0} \\}\]
Here the $d^i:\Delta^0\to\Delta^1$ are the coface maps. Elementary homotopy $\sim_e$ is a reflexive and symmetric relation, but it is not transitive ---the concatenation of polynomial homotopies is usually not polynomial. We let $\sim$ be the transitive closure of $\sim_e$ and call $f_0$ and $f_1$ \emph{(polynomially) homotopic} if $f_0\sim f_1$. It is easily shown that $f_0\sim f_1$ iff there exist $r\in\N$ and $f:A\to B^{\sd^r\Delta^1}$ such that the following diagrams commute:
\[\xymatrixcolsep{3em}\xymatrixrowsep{2em}\xymatrix{A\ar[d]_-{f_i}\ar[r]^-{f} & B^{\sd^r\Delta^1}\ar[d]^-{(d^i)^*} \\
	B\ar[r]^-{\cong} & B^{\sd^r\Delta^0} \\}\]
It turns out that $\sim$ is compatible with composition; see \cite{gersten}*{Lemma 1.1} for details. Thus, we have a category $[\Algl]$ whose objects are $\ell$-algebras and whose hom-sets are given by $[A,B]:=\Hom_\Algl(A, B)/\hspace{-0.2em}\sim$. We also have an obvious functor $\Algl\to [\Algl]$.

\begin{defi}[\cite{cortho}*{Definition 3.1.1}]Let $(A,I)$ and $(B,J)$ be two directed diagrams in $\Algl$ and let $f,g\in\Hom_{\Algl^\ind}\left((A,I),(B,J)\right)$. We call $f$ and $g$ \emph{homotopic} if they correspond to the same morphism upon applying the functor $\Algl^\ind\to [\Algl]^\ind$. We also write:
	\[[A_\bul,B_\bul]:=\Hom_{[\Algl]^\ind}\left((A,I),(B,J)\right)=\lim_{i\in I}\colim_{j\in J}[A_i,B_j]\]
\end{defi}

\subsection{Functions vanishing on a subset}\label{sec:vanish}

A \emph{simplicial pair} is a pair $(K,L)$ where $K$ is a simplicial set and $L\subseteq K$ is a simplicial subset. A \emph{morphism of pairs} $f:(K',L')\to (K,L)$ is a morphism of simplicial sets $f:K'\to K$ such that $f(L')\subseteq L$. A simplicial pair $(K,L)$ is \emph{finite} if $K$ is a finite simplicial set. We will only consider finite simplicial pairs, omitting the word ``finite'' from now on. Let $(K,L)$ be a  simplicial pair, let $B\in\Algl$ and let $r\geq 0$. Put:
\[B_r^{(K,L)}:=\ker\left(B^{\sd^rK}\to B^{\sd^rL}\right)\]
The last vertex map induces morphisms $B^{(K,L)}_r\to B^{(K,L)}_{r+1}$ and we usually consider $B^{(K,L)}_\bul$ as a  directed diagram in $\Algl$:
\[B^{(K,L)}_\bul:B^{(K,L)}_0\to B^{(K,L)}_1\to B^{(K,L)}_2\to\cdots\]
Notice that a morphism $f:(K',L')\to (K,L)$ induces a morphism $f^*:B^{(K,L)}_\bul\to B^{(K',L')}_\bul$ of $\Zo$-diagrams.

\begin{lem}[cf. \cite{cortho}*{Proposition 3.1.3}]\label{lem:simppair}
	Let $(K,L)$ be a simplicial pair and let $B\in\Algl$. Then $\Z^{(K,L)}_r$ is a free abelian group and there is a natural $\ell$-algebra isomorphism:
	\begin{equation}\label{eq:natiso1}B\otimes \Z^{(K,L)}_r\overset{\cong}\to B^{(K,L)}_r\end{equation}
\end{lem}
\begin{proof}
	The following sequence is exact by definition of $\Z^{(K,L)}_r$ and \cite{cortho}*{Lemma 3.1.2}:
	\begin{equation}\label{eq:sec}\xymatrix{0\ar[r] & \Z^{(K,L)}_r\ar[r] & \Z^{\sd^rK}\ar[r] & \Z^{\sd^rL}\ar[r] & 0 \\}\end{equation}
	The group $\Z^{\sd^rL}$ is free abelian by \cite{cortho}*{Proposition 3.1.3} and thus the sequence \eqref{eq:sec} splits. It follows that $\Z^{(K,L)}_r$ is free because it is a direct summand of the free abelian group $\Z^{\sd^rK}$. Moreover, the following sequence is exact:
	\[\xymatrix{0\ar[r] & B\otimes\Z^{(K,L)}_r\ar[r] & B\otimes\Z^{\sd^rK}\ar[r] & B\otimes\Z^{\sd^rL}\ar[r] & 0 \\}\]
	To finish the proof we identify $B\otimes \Z^{\sd^rK}\overset{\cong}\to B^{\sd^rK}$ using the natural isomorphism of \cite{cortho}*{Proposition 3.1.3}.
\end{proof}

\begin{exa}\label{exa:fSn}
	Following \cite{garku}*{Section 7.2}, we will write $B^{\fS_n}_\bul$ instead of $B^{(I^n,\partial I^n)}_\bul$. Note that $B^{\fS_0}_\bul$ is the constant $\Zo$-diagram $B$.
\end{exa}

\section{Main theorem}\label{sec:main}

\subsection{Multiplication morphisms}\label{sec:multi}

Let $(K,L)$ and $(K',L')$ be simplicial pairs. It follows from Lemma \ref{lem:simppair} that $\Z_r^{(K,L)}\otimes \Z_s^{(K',L')}$ identifies with a subring of $\Z^{\sd^rK}\otimes\Z^{\sd^sK'}$. Let $\mu^{K,K'}$ be the composite of the following ring homomorphisms:
\[\xymatrixcolsep{7em}\xymatrix{\Z^{\sd^rK}\otimes\Z^{\sd^sK'}\ar@{..>}[d]_-{\mu^{K,K'}}\ar[r]^-{(\gamma^s)^*\otimes(\gamma^r)^*} & \Z^{\sd^{r+s}K}\otimes \Z^{\sd^{r+s}K'}\ar[d]^-{(\pr_1)^*\otimes (\pr_2)^*} \\
	\Z^{\sd^{r+s}(K\times K')} & \Z^{\sd^{r+s}(K\times K')}\otimes \Z^{\sd^{r+s}(K\times K')}\ar[l]_-{\mult} \\ }\]
Here $\gamma^j$ is the iterated last vertex map defined in section \ref{subsec:lastvertex}, $\pr_i$ is the projection of the direct product into its $i$-th factor and $\mult$ is the map described in Remark \ref{rem:mult}.

\begin{lem}\label{lem:multi}
	The morphism $\mu^{K,K'}$ defined above induces a ring homomorphism $\mu^{(K,L),(K',L')}$ that fits into the following diagram:
	\[\xymatrixcolsep{5em}\xymatrix{\Z^{(K,L)}_r\otimes \Z^{(K',L')}_s\ar[r]^-{\inc}\ar[d]_-{\mu^{(K,L),(K',L')}} & \Z^K_r\otimes \Z^{K'}_s\ar[d]^-{\mu^{K,K'}} \\
	\Z^{(K\times K',(K\times L')\cup (L\times K'))}_{r+s}\ar[r]^-{\inc} & \Z^{K\times K'}_{r+s}}\]
	Moreover, $\mu^{(K,L),(K',L')}$ is natural in both variables with respect to morphisms of simplicial pairs. We call $\mu^{(K,L),(K',L')}$ a \emph{multiplication morphism}.
\end{lem}
\begin{proof}
	Let $\varepsilon$ be the restriction of $\mu^{K,K'}$ to $\Z^{(K,L)}_r\otimes\Z^{(K',L')}_s$; we have to show that $\varepsilon$ is zero when composed with the morphism:
	\[\Z^{\sd^{r+s}(K\times K')}\to \Z^{\sd^{r+s}((K\times L')\cup (L\times K'))}\]
	Since the functor $\Z^{\sd^{r+s}(?)}:\S\to \Alg_\Z^{\op}$ commutes with colimits, it will be enough to show that $\varepsilon$ is zero when composed with the projections to $\Z^{\sd^{r+s}(K\times L')}$ and to $\Z^{\sd^{r+s}(L\times K')}$; this is a straightforward check. For example, the following commutative diagram shows that $\varepsilon$ is zero when composed with the projection to $\Z^{\sd^{r+s}(L\times K')}$; we write $i$ for the inclusion $L\subseteq K$.
	\[\xymatrixcolsep{3em}\begin{gathered}\xymatrix{\Z^{(K,L)}_r\otimes\Z^{(K',L')}_s \ar[d]\ar@/^/[dr]^-{0} & \\
		\Z^{\sd^rK}\otimes \Z^{\sd^sK'}\ar[r]^-{i^*\otimes 1}\ar[d]_-{(\gamma^s)^*\otimes(\gamma^r)^*} & \Z^{\sd^rL}\otimes \Z^{\sd^sK'}\ar[d]^-{(\gamma^s)^*\otimes(\gamma^r)^*} \\
		\Z^{\sd^{r+s}K}\otimes \Z^{\sd^{r+s}K'}\ar[r]^-{i^*\otimes 1}\ar[d]_-{(\pr_1)^*\otimes (\pr_2)^*} & \Z^{\sd^{r+s}L}\otimes \Z^{\sd^{r+s}K'}\ar[d]^-{(\pr_1)^*\otimes (\pr_2)^*} \\
		\Z^{\sd^{r+s}(K\times K')}\otimes \Z^{\sd^{r+s}(K\times K')}\ar[r]^-{i^*\otimes i^*}\ar[d]_-{\mult} & \Z^{\sd^{r+s}(L\times K')}\otimes \Z^{\sd^{r+s}(L\times K')}\ar[d]^-{\mult} \\
		\Z^{\sd^{r+s}(K\times K')}\ar[r]^-{i^*} & \Z^{\sd^{r+s}(L\times K')} \\}\end{gathered}\]
	The assertion about naturality is clear.\end{proof}

\begin{rem}\label{rem:mular}
	We can consider $\Z^{(K,L)}_\bul\otimes\Z^{(K',L')}_\bul$ as a directed diagram of rings indexed over $\Zo\times \Zo$. Let $\theta:\Zo\times\Zo\to\Zo$ be defined by $\theta(r,s)=r+s$; it is clear that $\theta$ is a functor. Then the morphisms of Lemma \ref{lem:multi} assemble into a morphism in $\lar{\Alg_\Z}$:
	\[\xymatrix{\left(\mu^{(K,L),(K',L')},\theta\right):\Z^{(K,L)}_\bul\otimes\Z^{(K',L')}_\bul\ar[r] & \Z^{(K\times K',(K\times L')\cup (L\times K'))}_\bul}\]
	We will often think of $\mu^{(K,L),(K',L')}$ in this way, omitting $\theta$ from the notation.
\end{rem}

\begin{rem}\label{rem:mub}
	Upon tensoring $\mu^{(K,L),(K',L')}$ with an $\ell$-algebra $B$ and using \eqref{eq:natiso1} we obtain an $\ell$-algebra homomorphism:
	\[\xymatrix{\mu^{(K,L),(K',L')}_B:\left(B^{(K,L)}_r\right)^{(K',L')}_s\ar[r] & B^{(K\times K', (K\times L')\cup(L\times K'))}_{r+s}}\]
	This morphism is obviously natural with respect to morphisms of simplicial pairs and with respecto to $\ell$-algebra homomorphisms. Again, we can think of it as a morphism in $\lar{\Algl}$. It is easily verified that this morphism is associative in the obvious way.
\end{rem}

\begin{exa}\label{exa:muhtpyequiv}
	For any $n\geq 0$ and any $B\in\Algl$ we have a morphism $\iota:B\to B^{\Delta^n}$ induced by $\Delta^n\to\ast$. It is well known that $\iota$ is a homotopy equivalence, as we proceed to explain. Let $v:B^{\Delta^n}\to B$ be the restriction to the $0$-simplex $0$. Explicitely, we have $v(t_i)=0$ for $i>0$ and $v(t_0)=1$. It is easily verified that $v\circ\iota=\id_B$. Now let $H:B^{\Delta^p}\to B^{\Delta^n}[u]$ be the elementary homotopy defined by $H(t_i)=ut_i$ for $i>0$ and $H(t_0)=t_0+(1-u)(t_1+\cdots +t_n)$. We have $\ev_1\circ H=\id_{B^{\Delta^n}}$ and $\ev_0\circ H=\iota\circ v$. This shows that $\iota\circ v=\id_{B^{\Delta^n}}$ in $[\Algl]$.
	
	The homotopy $H$ constructed above is natural with respect to the inclusion of faces of $\Delta^n$ that contain the $0$-simplex $0$. More precisely: if $f:[m]\to [n]$ is an injective order-preserving map such that $f(0)=0$, then the following diagram commutes:
	\[\xymatrix{B^{\Delta^n}\ar[d]_-{f^*}\ar[r]^-{H} & B^{\Delta^n}[u]\ar[d]^-{f^*[u]} \\
		B^{\Delta^m}\ar[r]^-{H} & B^{\Delta^m}[u] }\]
	
	Now let $p,q\geq 0$. Recall from the proof of \cite{hovey}*{Lemma 3.1.8} that the simplices of $\Delta^p\times\Delta^q$ can be identified with the chains  in $[p]\times[q]$ with the product order. The nondegenerate $(p+q)$-simplices of $\Delta^p\times\Delta^q$ are identified with the maximal chains in $[p]\times[q]$; there are exactly $\binom{p+q}{p}$ of these. Following \cite{hovey}, let $c(i)$ for $1\leq i\leq \binom{p+q}{p}$ be the complete list of maximal chains of $[p]\times[q]$. Then $\Delta^p\times\Delta^q$ is the coequalizer in $\S$ of the two natural morphisms of simplicial sets $f$ and $g$ induced by the inclusions $c(i)\cap c(j)\subseteq c(i)$ and $c(i)\cap c(j)\subseteq c(j)$ respectively:
	\[\xymatrixcolsep{3em}\xymatrix{\displaystyle f,g:\coprod_{1\leq i<j\leq\binom{p+q}{p}}\Delta^{n_{c(i)\cap c(j)}}\ar@<-.5ex>[r] \ar@<.5ex>[r] & \displaystyle\coprod_{1\leq i\leq\binom{p+q}{p}}\Delta^{n_{c(i)}}}\]
	Here $n_c$ is the number of edges in $c$; that is, the dimension of the nondegenerate simplex corresponding to $c$. Since $B^?:\S^\op\to\Algl$ preserves limits, it follows that $B^{\Delta^p\times\Delta^q}$ is the equalizer of the following diagram in $\Algl$:
	\begin{equation}\label{eq:Hpq1}\xymatrixcolsep{3em}\xymatrix{\displaystyle f^*,g^*:\displaystyle\prod_{1\leq i\leq\binom{p+q}{p}}B^{\Delta^{n_{c(i)}}}\ar@<-.5ex>[r] \ar@<.5ex>[r] & \displaystyle\prod_{1\leq i<j\leq\binom{p+q}{p}}B^{\Delta^{n_{c(i)\cap c(j)}}}}\end{equation}
	Moreover, since $\Z[u]$ is a flat ring, $?\otimes\Z[u]$ preserves finite limits and $B^{\Delta^p\times\Delta^q}[u]$ is the equalizer of the following diagram:
	\begin{equation}\label{eq:Hpq2}\xymatrixcolsep{3em}\xymatrix{\displaystyle f^*[u],g^*[u]:\displaystyle\prod_{1\leq i\leq\binom{p+q}{p}}B^{\Delta^{n_{c(i)}}}[u]\ar@<-.5ex>[r] \ar@<.5ex>[r] & \displaystyle\prod_{1\leq i<j\leq\binom{p+q}{p}}B^{\Delta^{n_{c(i)\cap c(j)}}}[u]}\end{equation}
	Notice that every maximal chain of $[p]\times[q]$ starts at $(0,0)$. This implies, by the discussion above on the naturality of $H$, that the following diagram commutes for every $i$ and $j$:
	\[\xymatrix{B^{\Delta^{n_{c(i)}}}\ar[d]\ar[r]^-{H} & B^{\Delta^{n_{c(i)}}}[u]\ar[d] \\
		B^{\Delta^{n_{c(i)\cap c(j)}}}\ar[r]^-{H} & B^{\Delta^{n_{c(i)\cap c(j)}}}[u] }\]
	Then the homotopy $H$ on the different $B^{\Delta^{n_{c(i)}}}$ gives a morphism of diagrams from \eqref{eq:Hpq1} to \eqref{eq:Hpq2} that induces $H:B^{\Delta^p\times\Delta^q}\to B^{\Delta^p\times\Delta^q}[u]$. Let $\iota:B\to B^{\Delta^p\times\Delta^q}$ be the morphism induced by $\Delta^p\times\Delta^q\to \ast$ and let $v:B^{\Delta^p\times\Delta^q}\to B$ be the restriction to the $0$-simplex $(0,0)$. It is easily verified that $\ev_1\circ H$ is the identity of $B^{\Delta^p\times\Delta^q}$ and that $\ev_0\circ H=\iota\circ v$; this shows that $\iota$ is a homotopy equivalence.
	
	Finally, consider the following commutative diagram. Since each $\iota$ is a homotopy equivalence, it follows that $\mu^{\Delta^p,\Delta^q}:(B^{\Delta^p})^{\Delta^q}\to B^{\Delta^p\times\Delta^q}$ is a homotopy equivalence too.
	\[\xymatrix{B^{\Delta^p}\ar[r]^-{\iota} & (B^{\Delta^p})^{\Delta^q}\ar[d]^-{\mu} \\
		B\ar[r]^-{\iota}\ar[u]^-{\iota} & B^{\Delta^p\times\Delta^q}}\]
	
	The author does not know whether $\mu^{K,K'}:(B^K)^{K'}\to B^{K\times K'}$ is a homotopy equivalence for general $K$ and $K'$. It is true, though, that $\mu^{K,K'}$ induces an isomorphism in any homotopy invariant and excisive homology theory, as we explain below.
\end{exa}

\begin{exa}
	Let $\tria$ be a triangulated category with desuspension $\Omega$ and let $j:\Algl\to\tria$ be a functor satisfying properties (H) and (E). We will show that $j\left(\mu^{K,K'}:(B^K)^{K'}\to B^{K\times K'}\right)$ is an isomorphism, for any $K$, $K'$ and $B$.
	
	Fix an object $U$ of $\tria$ and consider a short exact sequence of $\ell$-algebras
	\[\scrE:\xymatrix{A'\ar[r] & A\ar[r] & A''\\}\]
	that splits as a sequence of $\ell$-modules. To alleviate notation, we will write $\tria^U_n(A)$ instead of $\Hom_\tria(U,\Omega^nj(A))$, for any $n\in\Z$. By property (E), we have a distinguished triangle
	\[\triangle_\scrE:\xymatrix{\Omega j(A'')\ar[r] & j(A')\ar[r] & j(A)\ar[r] & j(A'')\\}\]
	that induces a long exact sequence of groups as follows:
	\[\xymatrix{\ar@{..>}[r]& \tria_{n+1}^U(A'')\ar[r] & \tria_n^U(A')\ar[r] & \tria_n^U(A)\ar[r] & \tria_n^U(A'')\ar@{..>}[r] & \\}\]
	This sequence is moreover natural in $\scrE$. Now consider a cartesian square of $\ell$-algebras where the horizontal morphisms are split surjections of $\ell$-modules:
	\[\xymatrix{A\ar[r]\ar[d] & A'\ar[d] \\
	A''\ar[r] & A'''}\]
	Proceeding as in the proof of \cite{ralf}*{Theorem 2.41}, we get an exact Mayer-Vietoris sequence:
	\[\xymatrix{\ar@{..>}[r]& \tria_{n+1}^U(A''')\ar[r] & \tria_n^U(A)\ar[r] & \tria_n^U(A')\oplus \tria_n^U(A'')\ar[r] & \tria_n^U(A''')\ar@{..>}[r] & \\}\]
	This sequence is natural with respect to morphisms of squares.
	
	Let $q\geq 0$. We will show that $j(\mu^{K,\Delta^q})$ is an isomorphism by induction on the dimension of $K$. The case $\dim K=0$ follows from the facts that $j$ preserves finite products and that $\mu^{\Delta^0,\Delta^q}$ is an $\ell$-algebra isomorphism. Now let $n\geq 0$ and suppose $j(\mu^{K',\Delta^q})$ is an isomorphism for every finite $L$ with $\dim L\leq n$. If $\dim K=n+1$, we have a cocartesian square:
	\[\xymatrix{K & \sk^nK\ar[l] \\
	\coprod_1^r\Delta^{n+1}\ar[u] & \coprod_1^r\partial\Delta^{n+1}\ar[u]\ar[l]}\]
	Upon applying the functors $(B^{?})^{\Delta^q}$ and $B^{?\times\Delta^q}$ we get the following cartesian squares:
	\[\xymatrix{(B^K)^{\Delta^q}\ar[d]\ar[r] & (B^{\sk^nK})^{\Delta^q}\ar[d] \\
	\prod_1^r(B^{\Delta^{n+1}})^{\Delta^q}\ar[r] & \prod_1^r(B^{\partial\Delta^{n+1}})^{\Delta^q}}
	\hspace{1em}\xymatrix{B^{K\times\Delta^q}\ar[d]\ar[r] & B^{\sk^nK\times\Delta^q}\ar[d] \\
	\prod_1^rB^{\Delta^{n+1}\times\Delta^q}\ar[r] & \prod_1^rB^{\partial\Delta^{n+1}\times\Delta^q}}\]
	The horizontal morphisms in these diagrams are split surjections of $\ell$-modules. Indeed, the morphism $\Z^K\to \Z^{\sk^nK}$ is a split surjection of abelian groups since it is surjective by \cite{cortho}*{Lemma 3.1.2} and $\Z^{\sk^nK}$ is a free abelian group by \cite{cortho}*{Proposition 3.1.3}. Upon tensoring with $B$, we get that  $B^K\to B^{\sk^nK}$ is a split surjection of $\ell$-modules. Similar arguments apply to the remaining horizontal morphisms. By naturality of $\mu$, we have a morphism of squares from the square on the left to the square on the right; this induces a morphism of long exact Mayer-Vietoris sequences upon applying $\tria^U_*$. Note that $\tria^U_*(\mu^{\Delta^{n+1},\Delta^q})$ is an isomorphism by Example \ref{exa:muhtpyequiv}. It follows from the 5-lemma that
	\[(\mu^{K,\Delta^q})_*:\Hom_\tria(U,j((B^K)^{\Delta^q}))\to\Hom_\tria(U,j(B^{K\times\Delta^q}))\]
	is an isomorphism. Since $U$ is arbitrary, this implies that $j(\mu^{K,\Delta^q})$ is an isomorphism. Now we can show that $j(\mu^{K,K'})$ is an isomorphism by induction on the dimension of $K'$.
\end{exa}

\subsection{Main theorem}\label{sec:mainthm} Following \cite{garku}*{Section 7.2}, we put $\widetilde{B}^{\fS_n}_\bul:=B^{(I^n\times I,\partial I^n\times I)}_\bul$. The coface maps $d^i:\Delta^0\to I$ induce morphisms $(d^i)^*:\widetilde{B}^{\fS_n}_\bul\to B^{\fS_n}_\bul$.

\begin{lem}[Garkusha]\label{lem:haupt2}
	Let $f:A\to \widetilde{B}^{\fS_n}_r$ be an $\ell$-algebra homomorphism. Then the following composites are homotopic; i.e. they belong to the same class in $[A,B^{\fS_n}_r]$:
	\[A\overset{f}\longrightarrow \widetilde{B}^{\fS_n}_r\overset{(d^i)^*}\longrightarrow B^{\fS_n}_r\hspace{1em} (i=0,1)\]
\end{lem}
\begin{proof}
	This is \cite{garku}*{Hauptlemma (2)}.
\end{proof}

As noted by Garkusha in \cite{garku}, Lemma \ref{lem:haupt2} shows that if two morphisms are simplicially homotopic, then they are polynomially homotopic. The converse also holds, up to increasing the number of subdivisions:

\begin{lem}[cf. \cite{garku}*{Hauptlemma (3)}]\label{lem:haupt}
	Let $H:A\to (B^{\fS_n}_r)^{\sd^sI}$ be a homotopy between two $\ell$-algebra homomorphisms $A\to B^{\fS_n}_r$. Then there exists a morphism $\widetilde{H}:A\to \widetilde{B}^{\fS_n}_{r+s}$ in $\Algl$ such that the following diagram commutes for $i=0,1$:
	\[\xymatrix{A\ar[r]^-{H}\ar[d]_-{\widetilde{H}} & (B^{\fS_n}_r)^{\sd^sI}\ar[r]^-{(d^i)^*} & B^{\fS_n}_r\ar[d]^-{(\gamma^s)^*} \\
		\widetilde{B}^{\fS_n}_{r+s}\ar[rr]^-{(d^i)^*} & & B^{\fS_n}_{r+s} \\}\]
\end{lem}
\begin{proof}
	Let $\widetilde{H}$ be the composite:
	\[\xymatrixcolsep{7em}\xymatrix{A\ar[r]^-{H} & (B^{(I^n,\partial I^n)}_r)^{(I,\emptyset)}_s\ar[r]^-{\mu^{(I^n,\partial I^n),(I,\emptyset)}} & B^{(I^n\times I, (\partial I^n)\times I)}_{r+s}\\}\]
	It is immediate from the naturality of $\mu$ that $\widetilde{H}$ satisfies the required properties.
\end{proof}

\begin{thm}[cf. \cite{cortho}*{Theorem 3.3.2}]\label{thm:bij}
	For any pair of $\ell$-algebras $A$ and $B$ and any $n\geq 0$, there is a natural bijection:
	\begin{equation}\label{eq:bij}\pi_n\Hom_\Algl(A,B^\Delta)\cong [A,B^{\fS_n}_\bul]\end{equation}
\end{thm}
\begin{proof}
	We will show that $\pi_0\Omega^n\Ex^\infty\Hom_\Algl(A,B^\Delta)\cong [A,B^{\fS_n}_\bul]$. Consider $\Hom_\Algl(A,B^\Delta)$ as a simplicial set pointed at the zero morphism. For every $p\geq 0$ we have a pullback of sets:
	\begin{equation}\label{eq:pbzero}\begin{gathered}\xymatrix{\left(\Omega^n\Ex^\infty\Hom_\Algl(A,B^\Delta)\right)_p\ar[d]\ar[r] & \Map\left(I^n,\Ex^\infty\Hom_\Algl(A,B^\Delta)\right)_p\ar[d] \\
		\ast\ar[r] & \Map\left(\partial I^n,\Ex^\infty\Hom_\Algl(A,B^\Delta)\right)_p \\}\end{gathered}\end{equation}
	For a finite simplicial set $K$ we have:
	\begin{align*}\Map\left(K,\Ex^\infty\Hom_\Algl(A,B^\Delta)\right)_p&=\Hom_\S\left(K\times \Delta^p,\Ex^\infty\Hom_\Algl(A,B^\Delta)\right)\\
	&\cong\colim_r\Hom_\S\left(K\times\Delta^p,\Ex^r\Hom_\Algl(A,B^\Delta)\right)\\
	&\cong\colim_r\Hom_\S\left(\sd^r(K\times\Delta^p),\Hom_\Algl(A,B^\Delta)\right)\\
	&\cong\colim_r\Hom_\Algl\left(A,B^{\sd^r(K\times\Delta^p)}\right)\end{align*}
	It follows from these identifications, from \eqref{eq:pbzero} and from the fact that filtered colimits of sets commute with finite limits, that we have the following bijections:
	\begin{equation}\label{eq:bij0}\left(\Omega^n\Ex^\infty\Hom_\Algl(A,B^\Delta)\right)_0\cong \colim_r\Hom_\Algl(A,B^{\fS_n}_r)\end{equation}
	\begin{equation}\label{eq:bij1}\left(\Omega^n\Ex^\infty\Hom_\Algl(A,B^\Delta)\right)_1\cong \colim_r\Hom_\Algl(A,\widetilde{B}^{\fS_n}_r)\end{equation}
	Using \eqref{eq:bij0} we get a surjection:
	\[\left(\Omega^n\Ex^\infty\Hom_\Algl(A,B^\Delta)\right)_0\cong \colim_r\Hom_\Algl(A,B^{\fS_n}_r)\longrightarrow [A,B^{\fS_n}_\bul]\]
	We claim that this function induces the desired bijection. The fact that it factors through $\pi_0$ follows from the identification \eqref{eq:bij1} and Lemma \ref{lem:haupt2}. The injectivity of the induced function from $\pi_0$ follows from Lemma \ref{lem:haupt}.
\end{proof}

\begin{rem}\label{rem:groupstructure}
	Let $A$ and $B$ be two $\ell$-algebras and let $n\geq 1$. Endow the set $[A,B^{\fS_n}_\bul]$ with the group structure for which \eqref{eq:bij} is a group isomorphism. This group structure is abelian if $n\geq 2$. Moreover, if $f:A\to A'$ and $g:B\to B'$ are morphisms in $[\Algl]$, then the following functions are group homomorphisms:
	\[f^*:[A', B^{\fS_n}_\bul]\to [A, B^{\fS_n}_\bul]\]
	\[g_*:[A, B^{\fS_n}_\bul]\to [A, (B')^{\fS_n}_\bul]\]
\end{rem}

\begin{exa}\label{exa:groupinverse}
	Recall that $B^{\Delta^1}=B[t_0,t_1]/\langle 1-t_0-t_1\rangle$. Let $\omega$ be the automorphism of $B^{\Delta^1}$ defined by $\omega(t_0)=t_1$, $\omega(t_1)=t_0$; it is clear that $\omega$ induces an automorphism of $B^{\fS_1}_0=\ker(B^{\Delta^1}\to B^{\partial\Delta^1})$. Let $f:A\to B^{\fS_1}_0$ be an $\ell$-algebra homomorphism and let $[f]$ be its class in $[A,B^{\fS_1}_\bul]$. We claim that $[\omega\circ f]=[f]^{-1}\in[A,B^{\fS_1}_\bul]$. In order to prove this claim, we proceed to recall the definition of the group law $*$ in $\pi_1\Ex^\infty\Hom_\Algl(A,B^\Delta)$. Consider $f$ and $\omega\circ f$ as $1$-simplices of $\Ex^\infty\Hom_\Algl(A,B^\Delta)$ using the identification:
	\[\left(\Ex^\infty\Hom_\Algl(A,B^\Delta)\right)_1\cong\colim_r\Hom_\Algl(A,B^{\sd^r\Delta^1})\]
	According to \cite{goja}*{Section I.7}, if we find $\alpha\in\left(\Ex^\infty\Hom_\Algl(A,B^\Delta)\right)_2$ such that
	\begin{equation}\label{eq:condi}\left\{\begin{array}{l}d_0\alpha=\omega\circ f\\
	d_2\alpha=f\end{array}\right.\end{equation}
	then we have $[f]*[\omega\circ f]=[d_1\alpha]$. Let $\varphi:B^{\Delta^1}\to B^{\Delta^2}$ be the $\ell$-algebra homomorphism defined by $\varphi(t_0)=t_0+t_2$, $\varphi(t_1)=t_1$. Let $\alpha$ be the $2$-simplex of $\Ex^\infty\Hom_\Algl(A,B^\Delta)$ induced by the composite:
	\[A\overset{f}\to B^{\Delta^1}\overset{\varphi}\to B^{\Delta^2}\]
	It is easy to verify that the equations \eqref{eq:condi} hold and that $d_1\alpha$ is the zero path.
\end{exa}

\begin{exa}\label{exa:cmn}
	Let $A$ and $B$ be two $\ell$-algebras and let $m,n\geq 1$. Let $c:I^m\times I^n\overset{\cong}\to I^n\times I^m$ be the commutativity isomorphism; $c$ induces an isomorphism $c^*:B^{\fS_{n+m}}_\bul\to B^{\fS_{m+n}}_\bul$. We claim that the following function is multiplication by $(-1)^{mn}$:
	\[c^*:[A,B^{\fS_{n+m}}_\bul]\to [A,B^{\fS_{m+n}}_\bul]\]
	Indeed, this follows from Theorem \ref{thm:bij} and the well known fact that permuting two coordinates in $\Omega^{m+n}$ induces multiplication by $(-1)$ upon taking $\pi_0$.
\end{exa}

\section{Garkusha's Comparison Theorem A}\label{sec:spec}

Matrix-unstable algebraic $KK$-theory consists of a triangulated category $\DRF$ endowed with a functor $j:\Algl\to \DRF$ that satisfies (H) and (E), and is moreover universal with respect to these two properties. It was constructed by Garkusha in \cite{garkuuni}*{Section 2.4} by means of deriving a certain Brown category and then stabilizing the loop functor. Garkusha defined in \cite{garku} a space $\scrK(A,B)$ such that $\pi_0\scrK(A,B)\cong\Hom_\DRF(j(A),j(B))$, for any pair of $\ell$-algebras $A$ and $B$. In this section we apply Theorem \ref{thm:main} to give a simplified proof of \cite{garku}*{Comparison Theorem A}, where $\pi_0\scrK(A,B)$ is computed in terms of polynomial homotopy classes of morphisms.

\subsection{Extensions and classifying maps}\label{subsec:extensions}

Let $\Mod_\ell$ be the category of $\ell$-modules and write $F:\Algl\to \Mod_\ell$ for the forgetful functor. An \emph{extension} of $\ell$-algebras is a diagram
\begin{equation}\label{eq:extension}\scrE:\xymatrix{A\ar[r] & B\ar[r] & C\\}\end{equation}
in $\Algl$ that becomes a split short exact sequence upon applying $F$. A \emph{morphism of extensions} is a morphism of diagrams in $\Algl$. We often consider specific splittings for the extensions we work with and we sometimes write $(\scrE,s)$ to emphasize that we are considering an extension $\scrE$ with splitting $s$. Let $(\scrE,s)$ and $(\scrE',s')$ be two extensions with specified splittings; a \emph{strong morphism of extensions} $(\scrE',s')\to (\scrE,s)$ is a morphism of extensions $(\alpha,\beta,\gamma):\scrE'\to \scrE$ that is compatible with the splittings; i.e. such that the folowing diagram commutes:
\[\xymatrixrowsep{2em}\xymatrix{FB'\ar[d]_{F\beta} & FC'\ar[l]_-{s'}\ar[d]^-{F\gamma} \\
	FB & FC\ar[l]_-{s} \\}\]

The functor $F:\Algl\to\Mod_\ell$ admits a right adjoint $\tilT:\Mod_\ell\to\Algl$; see \cite{garku}*{Section 3} for details. Let $T$ be the composite functor $\tilT\circ F:\Algl\to\Algl$. Let $A\in\Algl$ and let $\eta_A:TA\to A$ be the counit of the adjunction. Notice that $F\eta_A$ is a retraction which has the unit map $\sigma_A:FA\to F\tilT(FA)=FTA$ as a section. Let $JA:=\ker \eta_A$. The \emph{universal extension} of $A$ is the extension:
\[\scrU_A:\xymatrix{JA\ar[r] & TA\ar[r]^-{\eta_A} & A\\}\]
We will always consider $\sigma_A$ as a splitting for $\scrU_A$.

\begin{prop}[cf. \cite{cortho}*{Proposition 4.4.1}]\label{lem:classexists}
	Let \eqref{eq:extension} be an extension with splitting $s$ and let $f:D\to C$ be a morphism in $\Algl$. Then there exists a unique strong morphism of extensions $\scrU_{D}\to(\scrE,s)$ extending $f$:
	\begin{equation}\label{eq:classexists}\begin{gathered}\xymatrix{\scrU_D\ar@{..>}[d]_-{\exists !} & JD\ar[r]\ar@{..>}[d]_-{\xi} & TD\ar@{..>}[d]\ar[r]^-{\eta_D} & D\ar[d]^-{f} \\
		(\scrE,s) & A\ar[r] & B\ar[r] & C \\}\end{gathered}\end{equation}
\end{prop}
\begin{proof}
	It follows easily from the adjointness of $\tilT$ and $F$.
\end{proof}

The morphism $\xi$ in \eqref{eq:classexists} is called the \emph{classifying map of $f$ with respect to the extension $(\scrE,s)$}. When $D=C$ and $f=\id_C$ we call $\xi$ the \emph{classifying map of $(\scrE,s)$}.

\begin{prop}[cf. \cite{cortho}*{Proposition 4.4.1}]\label{lem:classhmtp}
	In the hypothesis of Proposition \ref{lem:classexists}, the homotopy class of the classifying map $\xi$ does not depend upon the splitting $s$.
\end{prop}
\begin{proof}See, for example, \cite{garku}*{Section 3}.\end{proof}

Because of Proposition \ref{lem:classhmtp}, it makes sense to speak of the classifying map of \eqref{eq:extension} as a homotopy class $JC\to A$ without specifying a splitting for \eqref{eq:extension}.

\begin{lem}\label{lem:Jhtp}
	The functor $J:\Algl\to \Algl$ sends homotopic morphisms to homotopic morphisms. Thus, it defines a functor $J:[\Algl]\to[\Algl]$.
\end{lem}
\begin{proof}
	It is explained in \cite{cortho} in the discussion following \cite{cortho}*{Corollary 4.4.4.}.
\end{proof}

\subsection{Path extensions}\label{subsec:pathext} Let $B$ be an $\ell$-algebra and let $n,q\geq 0$. Put:
\[P(n,B)^q_\bul:=B^{(I^{n+1}\times\Delta^q,(I^n\times\{1\}\times\Delta^q)\cup(\partial I^n\times I\times\Delta^q))}_\bul\]
On the one hand, the composite $I^n\times\Delta^q\cong I^n\times\{0\}\times\Delta^q\subseteq I^{n+1}\times\Delta^q$ induces morphisms $p_{n,B}^q:P(n,B)^q_r\to B^{(I^n\times\Delta^q,\partial I^n\times\Delta^q)}_r$. On the other hand, we have inclusions $B^{(I^{n+1}\times\Delta^q,\partial I^{n+1}\times\Delta^q)}_r\subseteq P(n,B)^q_r$. We claim that the following diagram is an extension:
\[\scrP_{n,B}^q:\xymatrix{B^{(I^{n+1}\times\Delta^q,\partial I^{n+1}\times\Delta^q)}_r\ar[r]^-{\mathrm{incl}} & P(n,B)^q_r\ar[r]^-{p_{n,B}^q} & B^{(I^n\times\Delta^q,\partial I^n\times\Delta^q)}_r}\]
Exactness at $P(n,B)_r^q$ holds because the functors $B^{\sd^r(-)}:\S\to \Algl^\op$ preserve pushouts and we have:
\[\partial I^{n+1}\times \Delta^q=\left[ (I^n\times \{1\}\times \Delta^q)\cup(\partial I^n\times I\times \Delta^q)\right]\cup (I^n\times\{0\}\times \Delta^q)\]
Exactness at $B^{(I^{n+1}\times\Delta^q,\partial I^{n+1}\times\Delta^q)}_r$ follows from the fact that both this algebra and $P(n,B)_r^q$ are subalgebras of $B^{\sd^r(I^{n+1}\times \Delta^q)}$. A splitting of $p_{n,B}^q$ in the category of $\ell$-modules can be constructed as follows. Consider the element $t_0\in\Z^{\Delta^1}$; $t_0$ is actually in $\Z^{(I,\{1\})}_0$ since $d_0(t_0)=0$. Let $s_{n,B}^q$ be the composite:
\[\xymatrixrowsep{2em}\xymatrix{B^{(I^n\times\Delta^q,\partial I^n\times\Delta^q)}_r\ar@/_5pc/@{..>}[dddr]_-{s_{n,B}^q}\ar[r]^-{?\otimes t_0} & B^{(I^n\times\Delta^q,\partial I^n\times\Delta^q)}_r\otimes \Z^{(I,\{1\})}_0\ar[d]^-{\cong} \\
	& \left(B^{(I^n\times\Delta^q,\partial I^n\times\Delta^q)}_r\right)^{(I,\{1\})}_0\ar[d]^-{\mu} \\
	& B^{(I^n\times\Delta^q\times I,(I^n\times\Delta^q\times\{1\})\cup (\partial I^n\times\Delta^q\times I))}_r\ar[d]^-{\cong} \\
	& B^{(I^n\times I\times\Delta^q,(I^n\times\{1\}\times\Delta^q)\cup (\partial I^n\times I\times\Delta^q))}_r}\]
It is straightforward to check that $s_{n,B}^q$ is a section to $p_{n,B}^q$.

\begin{rem}\label{rem:functorialitiespathBnq}
	It is clear that the extensions $(\scrP^q_{n,B}, s_{n,B}^q)$ are:
	\begin{enumerate}[label=(\roman*)]
		\item natural in $B$ with respect to $\ell$-algebra homomorphisms;
		\item natural in $r$ with respect to the last vertex map;
		\item and natural in $q$ with respect to morphisms of ordinal numbers.
	\end{enumerate}
\end{rem}

\begin{exa}\label{exa:Lambda}
Let $A$ and $B$ be two $\ell$-algebras, let $n\geq 0$ and let $f:A\to B^{\fS_n}_r$ be an $\ell$-algebra homomorphism. By Proposition \ref{lem:classexists}, there exists a unique strong morphism of extensions $\scrU_A\to\scrP_{n,B}^0$ that extends $f$:
	\[\xymatrixcolsep{3em}\xymatrix{\scrU_A\ar@{..>}[d]_-{\exists !} & JA\ar[r]\ar@{..>}[d]_-{\Lambda^n(f)} & TA\ar@{..>}[d]\ar[r] & A\ar[d]^-{f} \\
		\scrP_{n,B}^0 & B^{\fS_{n+1}}_r\ar[r] & P(n,B)_r^0\ar[r] & B^{\fS_n}_r \\}\]
	We will write $\Lambda^n(f)$ for the classifying map of $f$ with respect to $\scrP_{n,B}^0$. Notice that:
\begin{equation}
\label{eq:LambdaJ}\Lambda^n(f)=\Lambda^n(\id_{B^{\fS_n}})\circ J(f)
\end{equation}
	Indeed, this follows from the uniqueness statement in Proposition \ref{lem:classexists} and the fact that the following diagram exhibits a strong morphism of extensions $\scrU_A\to \scrP_{n,B}^0$ that extends $f$:
		\[\xymatrixcolsep{3em}\xymatrix{\scrU_A\ar[d] & JA\ar[d]_-{J(f)}\ar[r] & TA\ar[r]\ar[d]^-{T(f)} & A\ar[d]^-{f} \\
			\scrU_{B^{\fS_n}_r}\ar[d] & J(B^{\fS_n}_r)\ar[d]_-{\Lambda^n(\id_{B^{\fS_n}})}\ar[r] & T(B^{\fS_n}_r)\ar[r]\ar[d] & B^{\fS_n}_r\ar[d]^-{\id} \\
			\scrP_{n,B}^0 & B^{\fS_{n+1}}_r\ar[r] & P(n,B)_r^0\ar[r] & B^{\fS_n}_r \\}\]
By \eqref{eq:LambdaJ} and Lemma \ref{lem:Jhtp}, we can consider $\Lambda^n$ as a function $\Lambda^n:[A,B^{\fS_n}_\bul]\to[JA, B^{\fS_{n+1}}_\bul]$.
\end{exa}

\begin{rem}
	Write $\Born$ for the category of bornological algebras; see \cite{ralf}*{Definition 2.5}. Cuntz-Meyer-Rosenberg constructed in \cite{ralf}*{Section 6.3} a triangulated category $\SigmaHo$ endowed with a functor $\Born\to\SigmaHo$ that is homotopy invariant and excisive in the bornological context, and is moreover universal with respect to these two properties; see \cite{ralf}*{Section 6.7}. The matrix-unstable algebraic $KK$-theory category $\DRF$ is the analogue of $\SigmaHo$ in the algebraic context. Garkusha proved in \cite{garku} that there is an isomorphism
	\begin{equation}\label{eq:idea}
	\Hom_\DRF(j(A),j(B))\cong \colim_n[J^nA, B^{\fS_n}_\bul],
	\end{equation}
	where the transition functions on the right hand side are the $\Lambda^n$ of Example \ref{exa:Lambda}. These functions $\Lambda^n:[J^nA,B^{\fS_n}_\bul]\to[J^{n+1}A, B^{\fS_{n+1}}_\bul]$ are the algebraic analogues of the morphism $\Lambda$ of \cite{ralf}*{Definition 6.23} that is used in \cite{ralf}*{Section 6.3} to define the hom-sets in $\SigmaHo$.
\end{rem}

\subsection{Matrix-unstable algebraic $KK$-theory space}\label{sec:kkspace} 
Let $A$ and $B$ be two $\ell$-algebras and let $n\geq 0$. From the proof of Theorem \ref{thm:bij}, it follows that there is a natural bijection:
\[\left(\Omega^n\Ex^\infty\Hom_\Algl(J^nA,B^\Delta)\right)_q\cong \colim_r\Hom_\Algl\left(J^nA,B^{(I^n\times\Delta^q, \partial I^n\times\Delta^q)}_r\right)\]
Let $f\in\Hom_\Algl(J^nA,B^{(I^n\times\Delta^q, \partial I^n\times\Delta^q)}_r)$ and define $\zeta^n(f)$ as the classifying map of $f$ with respect to the extension $\scrP_{n,B}^q$:
\[\xymatrix{\scrU_{J^nA}\ar@{..>}[d]_-{\exists !} & J^{n+1}A\ar[r]\ar@{..>}[d]_-{\zeta^n(f)} & TJ^nA\ar@{..>}[d]\ar[r] & J^nA\ar[d]^-{f} \\
	\scrP_{n,B}^q & B^{(I^{n+1}\times\Delta^q, \partial I^{n+1}\times\Delta^q)}_r\ar[r] & P(n,B)^q_r\ar[r] & B^{(I^n\times\Delta^q, \partial I^n\times\Delta^q)}_r \\}\]
It follows from Remark \ref{rem:functorialitiespathBnq} that this defines a morphism of simplicial sets:
\begin{equation}\label{eq:zetaene}\zeta^n:\Omega^n\Ex^\infty\Hom_\Algl(J^nA,B^\Delta)\to \Omega^{n+1}\Ex^\infty\Hom_\Algl(J^{n+1}A,B^\Delta)\end{equation}

\begin{defi}[Garkusha]
	Let $A$ and $B$ be two $\ell$-algebras. The \emph{matrix-unstable algebraic $KK$-theory space} of the pair $(A,B)$ is the simplicial set defined by
	\[\scrK(A,B):=\colim_n\Omega^n\Ex^\infty\Hom_\Algl(J^nA,B^\Delta),\]
	where the transition morphisms are the $\zeta^n$ defined in \eqref{eq:zetaene}.
\end{defi}

Note that $\scrK(A,B)$ is a fibrant simplicial set, since it is a filtering colimit of fibrant simplicial sets. This definition of $\scrK(A,B)$ is easily seen to be the same as the one given in \cite{garku}*{Section 4}.

\begin{thm}[\cite{garku}*{Comparison Theorem A}]\label{thm:comparisonthms1}
	For any pair of $\ell$-algebras $A$ and $B$ and any $m\geq 0$, there is a natural isomorphism
	\[\pi_m\scrK(A,B)\cong \colim_v[J^vA, B^{\fS_{m+v}}_\bul]\]
	where the transition functions on the right hand side are the $\Lambda^n$ of Example \ref{exa:Lambda}. 
\end{thm}
\begin{proof}
	Since $\pi_m\cong\pi_0\Omega^m$ commutes with filtered colimits, we have:
	\begin{align*}
	\pi_m\scrK(A,B)&\cong \colim_v\pi_0\Omega^m\Omega^v\Ex^\infty\Hom_\Algl(J^vA,B^\Delta) \\
	&\cong \colim_v\pi_0\Omega^{v+m}\Ex^\infty\Hom_\Algl(J^vA,B^\Delta) \\
	&\cong \colim_v[J^vA,B^{\fS_{v+m}}_\bul] && \text{(by Theorem \ref{thm:bij})}
	\end{align*}
	Notice that $\Omega^m\Omega^v\cong\Omega^{v+m}$ because of our conventions on iterated loop spaces; see section \ref{subsec:loop}. To finish the proof, we need to compare the function $\Lambda^{m+v}$ with:
	\[\pi_m\zeta^v:[J^vA,B^{\fS_{v+m}}_\bul]\to [J^{v+1}A,B^{\fS_{v+1+m}}_\bul]\]
	Let $c_{v,m}:I^v\times I^m\overset{\cong}\to I^m\times I^v$ be the commutativity isomorphism; $c_{v,m}$ induces an isomorphism $(c_{v,m})^*:B^{\fS_{m+v}}_r\to B^{\fS_{v+m}}_r$. It is straightforward to verify that the following squares commute:
	\[\xymatrixcolsep{3em}\xymatrix{[J^vA, B^{\fS_{v+m}}_\bul]\ar[r]^-{\pi_m\zeta^v} & [J^{v+1}A, B^{\fS_{v+1+m}}_\bul] \\
		[J^vA, B^{\fS_{m+v}}_\bul]\ar[r]^-{\Lambda^{m+v}}\ar[u]_-{\cong}^-{(c_{v,m})^*} & [J^{v+1}A, B^{\fS_{m+v+1}}_\bul]\ar[u]^-{\cong}_-{(c_{v+1,m})^*}}\]
	These squares assemble into a morphism of diagrams that, upon taking colimit in $v$, induces the desired isomorphism $\pi_m\scrK(A,B)\cong \colim_v[J^vA, B^{\fS_{m+v}}_\bul]$.
\end{proof}

\begin{rem}
	It is posible to mimic the definition of $\SigmaHo$ in \cite{ralf}*{Section 6.3} to give a new and more explicit construction of Garkusha's matrix-unstable algebraic $KK$-theory category $\DRF$. Indeed, we can take \eqref{eq:idea} as a \emph{definition} of the hom-sets in $\DRF$. Theorem \ref{thm:bij} provides $[J^nA, B^{\fS_n}_\bul]$ with the group structure needed to make sense of the signs that appear when defining the composition rule; see \cite{ralf}*{Lemmas 6.29 and 6.30}. The algebraic context is, however, a little more complicated than the bornological one. We will develop these ideas further in a future paper.
\end{rem}

\end{document}